\input ./style/arxiv-general.cfg
\documentclass[aap,MSNbibl,seceqn,dvips]{arximspdf}
\makeatletter
   \@ifpackageloaded{graphicx}{}{\usepackage{graphicx}}
\makeatother
\usepackage{url,breakurl}

\doi{10.1214/14-AAP1077}
\volume{25}
\issue{5}
\pubyear{2015}
\firstpage{3007}
\lastpage{3032}
\docsubty{FLA}

\makeatletter
\newcommand{\ds}{\displaystyle}
\newcommand{\Beta}{\operatorname{Beta}}
\newtheorem{theo}{Theorem}[section]
\newtheorem{cor}[theo]{Corollary}
\newtheorem{prop}[theo]{Proposition}
\newtheorem{proposition}[theo]{Proposition}
\newtheorem{lem}[theo]{Lemma}
\newproclaim{rem}[theo]{Remark}
\makeatother

\begin{document}
\begin{frontmatter}

\title{The fixation line in the ${\Lambda}$-coalescent}
\runtitle{The fixation line}

\begin{aug}
\author[A]{\fnms{Olivier}~\snm{H\'enard}\corref{}\thanksref{T1}\ead[label=e1]{o.henard@qmul.ac.uk}}
\runauthor{O. H\'enard}
\affiliation{Queen Mary University of London}
\address[A]{School of Mathematical Sciences\\
Queen Mary University of London\\
Mile End Road\\
London E1 4NS\\
United Kingdom \\
\printead{e1}}
\end{aug}
\thankstext{T1}{Supported by the DFG Priority Programme
SPP 1590 ``Probabilistic Structures in Evolution''.}

%
\received{\smonth{8} \syear{2013}}
%
\revised{\smonth{9} \syear{2014}}

%
\begin{abstract}
We define a Markov process in a forward population model
with backward genealogy given by the $\Lambda$-coalescent. This Markov process,
called the fixation line,
is related to the block counting process through its hitting times.
Two applications are discussed.
The probability that the $n$-coalescent is deeper than the $(n-1)$-coalescent is studied.
The distribution of the number of blocks in the last coalescence
of the $n$-$\operatorname{Beta}(2-\alpha,\alpha)$-coalescent
is proved to converge as $n \rightarrow\infty$, and the generating
function of
the limiting random variable
is computed.
\end{abstract}

%
\end{frontmatter}

\section{Introduction}
\label{secfixation}

The $n$-coalescent is a stochastic model for the genealogy of a
(haploid) population of $n$ individuals
backward in time. In this model, the individuals of the population are
identified with the
integers of the set $\{1,\ldots,n\}$, and the $n$-coalescent
takes its values in the partitions of $\{1,\ldots,n\}$.
A partition of $\{1,\ldots,n\}$ is composed of a certain number of
blocks, between $1$ and $n$.
The initial state of the $n$-coalescent is the partition in $n$ blocks,
that is, the partition in singletons,
$\{1 \}, \ldots, \{n\}$. Any particular set of $k$ blocks then merges
independently
in one block at rate given by
%
\begin{equation}
\label{ratecoalescent0} \int_{[0,1]} \Lambda(dx) x^{-2}
x^{k} (1-x)^{n-k},
\end{equation}
where $\Lambda(dx)$ is a probability measure on $[0,1]$.
After the first coalescence, again, any particular set of $k$ blocks
merges independently
in one block at rate given by~(\ref{ratecoalescent0}), with $n$ replaced
by the current
numbers of blocks. The procedure is then repeated, until
the process terminates at the partition in one single block,
$\{1, \ldots, n\}$.
\textit{The first motivation of this paper is to study the number of
blocks involved in the last coalescence}.

The interpretation of the model is the following: two integers are in
the same block of the partition at time $t \geq0$ in the
$n$-coalescent if the corresponding individuals found their common
ancestor at time $t$ backward in time.
At the time of the last coalescence, all the individuals found their
common ancestor.

The blocks in a partition of $\{1,\ldots,n\}$ may be ordered according
to their smallest element.
With this ordering, the $n$-coalescent is described by a family
$(B_i(t), t \geq0, i \in\mathbb{N})$, where $B_1(t)$ is the block containing
$1$ at time $t$, $B_2(t)$ is the block containing the smallest integer
not in $B_1(t)$ at time $t$
(if any) $\ldots\,$. The largest~$i$ such that $B_i(t)$ is nonempty is
the number of blocks in the $n$-coalescent,
denoted by $X^{n}(t)$. For instance, for the $7$-coalescent depicted on
Figure~\ref{picture00}, we have $X^7(s)=5$ and:
\begin{eqnarray*}
B_1(s)&=& \{1,2,5\}, \qquad B_2(s)=\{3\},\qquad
B_3(s)=\{4\}, \qquad B_4(s)=\{6\} \quad\mbox{and}\\
B_5(s) &=& \{7\}.
\end{eqnarray*}

Alternatively, we may view the $n$-coalescent as a family of
(coalescing) maps. For each $i \in\{1,\ldots,n\}$, and $t \geq0$,
there exists a unique integer $j$ such that $i \in B_{j}(t)$ and we
then set $i(t)=j$. The random map $t \mapsto i(t)$ is nonincreasing,
starts at $i(0)=i$ and terminates at $1$. Furthermore, if two functions
$i(t)$ started at different points meet, then they coincide at each
further time. Figure~\ref{picture00} describes the collection of the
maps $i(t)$ started at $i \in\{1,\ldots,7\}$. Notice that a function
$i(t)$ not only decreases when the block labelled $i(t)$ takes part in
a coalescence, but also when at least two blocks with lower label take
part in a coalescence.

\begin{figure}

\includegraphics{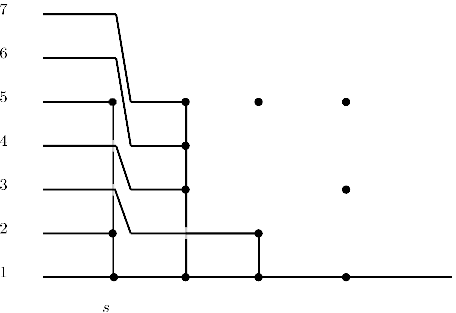}

\caption{A sample of the $7$-coalescent. The first coalescence,
at time $s$, has $j_1,j_2,j_3$ equal to $1,2,5$, respectively. The set
of records satisfies $\mathcal{T}\cap\{1, \ldots, 7\}=\{2,3\}$.}
\label{picture00}
\end{figure}

It is possible to couple the $n$-coalescents for distinct values of
$n$, in such a way that the $n$-coalescent is the restriction of the
$(n+1)$-coalescent to the first $n$ integers (this is shown by the
Poisson construction below). This coupling, that we will call the
natural coupling, allows to define the coalescent started from an
infinite number of blocks, simply called the coalescent, or the
$\Lambda$-coalescent if there is a need to stress on the measure $\Lambda(dx)$.

When adding more and more functions $t\rightarrow i(t)$ in the natural coupling,
the block counting process $(X^n(t), t \geq0)$ of the $n$-coalescent evolves,
and we may think of it as a wave moving to the right.
The motion to the right, measured by the depth $\tau_{1}^n= \inf\{ t\geq0, X^{n}(t)=1 \}$ of the $n$-coalescent,
is either a.s. bounded,
or a.s. unbounded---we come back to this fact in Section~\ref{subkeylemma}.
In both cases, we investigate the question of the existence
of a limiting shape (in distribution in the second case) for the wave
viewed from the right.
This amounts to study the time-reversal of the block counting process
for the coalescent
started from $n$ blocks as $n \rightarrow\infty$, a slight
elaboration on our
first motivation.

The depth of the $n$-coalescent
$\tau_{1}^n$ corresponds to the first time the $n$ blocks have merged
in $1$ block.
In the aforementioned natural coupling, we may consider the random
subset of the integers
\[
\mathcal{T}= \bigl\{n \geq2, \tau_{1}^n>
\tau_{1}^{n-1} \bigr\},
\]
that will be called the set of records: the integer $n$ belongs to the
set of records $\mathcal{T}$ when the function $i(t)$ started at $n$ reaches
$1$ at some later time than the functions $i(t)$ started at lower
values $i$ for $2 \leq i \leq n-1$,
thus establishing a new record. See Figure~\ref{picture00} for an illustration.
Since $\tau_{1}^{1}=0$ by definition, we have that $2 \in\mathcal{T}$ a.s.
In terms of population genetics,
the label of an individual corresponds to a record if its addition in
the sample modifies the most recent common ancestor
of the sample.
\textit{The study of the set of records is our second motivation}.

The lookdown model was introduced by Donnelly and Kurtz \cite{DO99}.
It is
essentially a time-reversal of the coalescent viewed as a family of
coalescing maps; see Figure~\ref{picture00}. Its construction echoes
the Poisson construction of the $n$-coalescent, and we first introduce
this construction.

Assume that $\Lambda\{0\}=0$. The construction starts with a Poisson
point measure
on $\mathbb R^+ \times(0,1]$, with intensity $dt \Lambda(dx) x^{-2}$.
To each atom $(t,x)$ of this random measure, we associate a random
subset
\[
J=\{j_1,j_2,j_3, \ldots\}
\]
of the set of integers $\mathbb N$ by sampling each integer
independently with the same probability $x$.
Then the blocks with labels in $J$ at time $t-$ coalesce in one block.
Notice that, among the (possibly infinitely many) atoms $(t,x)$ on a
finite time interval, only a finite number give rise to an effective
merge in the $n$-coalescent, and so we may distinguish a first
coalescence, a second one, etc.
The reader may check that this construction produces a Markov process
with transition rates given by (\ref{ratecoalescent0}).


The lookdown model starts with a family, indexed by a time $t \geq0$,
of countably many individuals distinguished by their integer-valued
level. The individual at time~$t$ at level $i \in\mathbb{N}$ is
denoted by
$(t,i)$. The lookdown model describes the \textit{genealogical
relationships} between the individuals at distinct times. Its
construction starts from the same Poisson point measure
on $\mathbb R^+ \times(0,1]$, with intensity $dt   \Lambda(dx) x^{-2}$.
A random set
\[
J= \{ j_1,j_2,j_3, \ldots\}
\]
is associated with each atom $(t,x)$ by sampling independently each
integer with the same probability $x$.
To each atom $(t,x)$, there corresponds a reproduction event and:
\begin{itemize}
\item for $j$ in $J$, the individual $(t,j)$
is a
child of
the individual $(t-, \min J)$. Notice that $J$, and therefore the set
of children, is infinite;
\item the other lineages are shifted upward, keeping the order
they had before the birth event: if
$k \notin J$, the individual $(t,k)$ is the (unique) child
of the individual $(t-,k-k')$, where
$k'= (\operatorname{Card}\{ J \cap\{1, \ldots,k\} \} -1) \vee0$,
see Figure~\ref{picture01}.
\end{itemize}
The real number $x$ will be called the asymptotic frequency of the
reproduction event.
The ancestral lineage of the individual $(t,i)$ is the line composed of
the individuals $((s,j(s)), 0 \leq s \leq t)$ where $j(s)$ is the level
of the ancestor of the individual $(t,i)$ at time $s$.
Figure~\ref{picture01} displays the collection of the ancestral lineages.
The connection between the two models is the following.
For each $t \in\mathbb R^+$, the ancestral lineages of the individuals
at the $n$ lowest levels at time $t$
define a process valued in the partitions of $\{1,\ldots, n\}$:
$i$ and $j$ are in the same block at time $s$, $0 \leq s \leq t$, if
the individuals $(t,i)$ and $(t,j)$ share
a common ancestor at time $t-s$.
It should be clear from Figure~\ref{picture01} that this partition
valued process has the law of (the restriction
to $[0,t]$) of the $n$-coalescent.

\begin{figure}

\includegraphics{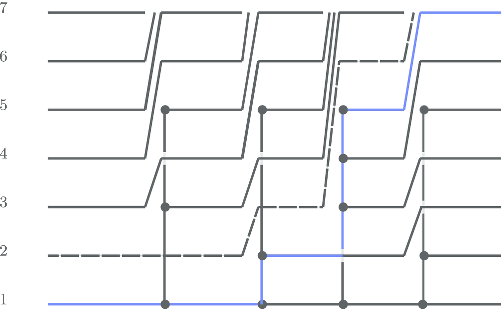}

\caption{The ancestral lineages in a look-down graph restricted to its
first $7$ levels.
The first reproduction event has $J= \{1,3,5,\ldots\}$.
The fixation line started at level $1$ at time $0$ is blue.
The dotted line above is the translation by one level of the fixation line.}
\label{picture01}
\end{figure}

The levels of the offspring at time $t \geq0$ of the individual
$(0,2)$ form a subset of $\mathbb N$,
the minimal element we define to be $L_1(t)+1$.
If the subset of $\mathbb N$ is empty, then we set $L_1(t)=\infty$.
The collection of the random variables $(L_1(t), t \geq0)$ builds a
nondecreasing process called \textit{the fixation line}, that is, the
blue line
in Figure~\ref{picture01}.
The shift by $1$ in the definition is for technical reason.
Alternatively, the levels of the offspring at time $t \geq0$ of the
individual $(0,1)$ form a subset of $\mathbb N$,
the connected component including $1$ is
$\{1, \ldots, L_1(t) \}$ [with, again, the convention that
$L_1(t)=+\infty$ if this subset is $\mathbb{N}$].
At the time the fixation line $L_1$
reaches level~$n$, the whole population of individuals at level
$1,\ldots,n$ consists
of offspring of the individual $(0,1)$, an event called fixation in
population genetics, whence the name
fixation line. The link between the fixation line and the set of
records $\mathcal{T}$ is the following:
for each $t \geq0$, $L_1(t)+1$ belongs to the set of records $\mathcal{T}$ for
the coalescent describing the genealogy of the individuals at time $t$.
See Figure~\ref{picture01} for an example.

This is a review of the literature:
the origin of the fixation line may be traced back to Pfaffelhuber and
Wakolbinger \cite{PW06} in the Kingman case
$\Lambda(dx)=\delta_0(dx)$. For general $\Lambda(dx)$, it appeared in
Labb\'e \cite{L11} and in \cite{H13}.
The coalescent we will focus on is the $\Lambda$-coalescent, that was
introduced independently and simultaneously
by Donnelly and Kurtz \cite{DO99}, Pitman \cite{PI99} and Sagitov
\cite{SA99}. Lecture notes have been written by Berestycki \cite{BE09} and
Bertoin \cite{BE06}, and the research area has been recently surveyed
in Gnedin, Iksanov and Marynynch \cite{GIM14}.
We find the fixation line useful in studying two random quantities
defined in the coalescent:
the set of records $\mathcal{T}$ and the number of blocks implied in
the last
coalescence of the $n$-coalescent.
An integer $n$ is a record when the corresponding external branch in
the coalescent tree has
depth equal to that of the $n$-coalescent tree.
In that sense, our analysis of the set of records may be seen as an
atypical view on the intensively studied
external branches; see Caliebe et al. \cite{CN07}, Dhersin and M\"ohle
\cite{DM13} and the references therein.
The numbers of blocks implied in the last coalescence, as well as
the closely related hitting probabilities of the block counting processes,
are quantities that relate to the coalescent tree near the root.
This part of the tree is difficult to grasp from the standard
construction of the coalescent.
Original techniques have been developed in the papers
\cite{AD13+,AD13,BBS08,GM05,K12,MO13}
to circumvent this difficulty.
Among these papers, the ones with the closest objectives to ours are
Abraham and Delmas \mbox{\cite{AD13+,AD13}}
and Goldschmidt and Martin \cite{GM05} that both
use a connection to a specific class of random trees.
The papers \cite{GM05,MO13} are concerned with
the Bolthausen--Sznitman coalescent, which is the $\Beta(2-\alpha,\alpha)$-coalescent
with $\alpha=1$, whereas
\cite{AD13+,AD13} deal with the $\Beta(2-\alpha,\alpha)$-coalescent
for $\alpha\in(0,1/2]$.
The approach that we propose allows us to deal with the
whole family of parameters $\alpha\in(0,2)$.
It should be
pointed out that the reason that makes it possible
to derive the distribution of the number of blocks in the last
coalescence of the $\Beta(2-\alpha,\alpha)$-coalescent is a connection
that exists between this coalescent and the trees with an
$\alpha$-stable branching mechanism, although this connection will not be
explicitly mentioned nor used. The connection originated in \cite{BI05}
and is usually stated in the framework of continuous state branching
processes. Here, we give a more ``discrete'' account. We point out this
is yet another connection, with trees that have a $1/(1-\alpha)$-stable
branching mechanism for $\alpha\in(0, 1/2]$, that is used by Abraham
and Delmas \mbox{\cite{AD13+,AD13}}.
Short after the preprint for this paper appeared, M\"ohle extended in
\cite{MO13+} the analytic method of \cite{MO13} to
the $\Beta(2-\alpha,\alpha)$-coalescent, and obtained some of the
results presented in this paper.
We warmly invite the reader to consult this paper as a parallel reference.

This is the organization of the paper:
Section~\ref{sec1} contains: a key lemma relating the depth of the $n$-coalescent to
the hitting times of the fixation line; see Lemma~\ref{tau-alpha}; the computation of the transition rates of the fixation line (these
two first points are in the general $\Lambda$ setting); the factorization of the rates in the $\Beta(2-\alpha,\alpha)$-coalescent.
These three ingredients combine in Section~\ref{sec2} to answer the
questions introduced above.
Namely, we first compute in Section~\ref{sec21} the probability for an integer
to be in
the random set of records $\mathcal{T}$.
Second, we characterize in Section~\ref{sec22} the time-reversal of the block
counting process.
The problem reduces to an analysis of
the number of blocks implied in the last coalescence and our main
result, Theorem~\ref{propconv-Delta},
is a limit theorem for the law of this random variable in the case of
the $n$-$\Beta(2-\alpha, \alpha)$-coalescent.
Corollary~\ref{corhitting} is a reformulation in term of
the hitting probabilities of an integer $j$ by the block counting
process (the $j \rightarrow\infty$ asymptotics of these hitting probabilities
are also computed).
Last, we connect the (discrete) Neveu branching process
to the Bolthausen--Sznitman coalescent, and deduce the fluctuations of
the depth of this coalescent.

\section{The fixation line}
\label{sec1}

\textit{Assumption}: we will assume throughout that the probability
measure $\Lambda(dx)$ gives no mass to the singletons
$0$ and $1$,
\[
\Lambda\{0\}= \Lambda\{1\}=0.
\]
The assumption on $\Lambda\{0\}$ allows to rely on the simple Poisson
construction of the coalescent mentioned in the
\hyperref[secfixation]{Introduction}
without struggling to include binary
coagulations. This being said, most of the results are still valid for
a probability measure $\Lambda(dx)$ with an atom at $0$.
The assumption on $\Lambda\{1\}$ avoids an uninteresting case.

\subsection{A key lemma on hitting times, and coming down from infinity}
\label{subkeylemma}
We first come back to the definition of the fixation line and
generalize it slightly by allowing the fixation line
to be started at an arbitrary integer.
Fix an integer $j$ and consider the set of individuals $(0,i)$ at time
$0$ at level $i$, for $i \in\{1,\ldots,j\}$.
The offspring of this set of individuals at time $t \geq0$ forms a
subset of $\mathbb N$
in the lookdown model mentioned in the \hyperref[secfixation]{Introduction}, the connected
component including $1$ we denote by
\[
\bigl\{1, \ldots, L_j(t) \bigr\},
\]
with $L_j(t)=\infty$ if this set is $\mathbb{N}$.
Alternatively, the offspring at time $t$ of the \textit{single}
individual $(0,j+1)$
forms a subset of $\mathbb N$, the \textit{smallest} element of which
is \mbox{$L_j(t)+1$}.

For $j$ and $n$ integers, we set
%
\begin{equation}
\label{eqdeftaualpha} \tau_{j}^n=\inf\bigl\{t \geq0,
X^n(t) \leq j\bigr\}, \qquad \alpha_{j}^n=\inf
\bigl\{t\geq0, L_j(t) \geq n\bigr\}
\end{equation}
to denote
the (partial) depth of the $n$-coalescent
and the hitting time of the fixation line,
respectively.
We now state our key lemma. In fact, the whole paper may be seen as a
digression on this relation.

\begin{lem}
\label{tau-alpha}
Let $1 \leq j \leq n$. The two random variables $\tau_{j}^n$ and
$\alpha_{j}^n$ have the same distribution.
\end{lem}

\begin{pf}
Fix $t>0$. It is enough to observe that the events
\[
\bigl\{ \tau_{j}^n >t \bigr\} \quad \mbox{and}\quad \bigl
\{ \alpha_{j}^n >t \bigr\}
\]
are equal in the coupling of a coalescent and a fixation line provided
by the lookdown model.
More precisely, we compare the fixation line started at level $j$ at
time~$0$ and the
coalescent describing the
backward genealogy of the $n$ lowest level
individuals at time $t$.
For the inclusion, observe that,
if $\tau_{j}^n >t$ for the coalescent, then the fixation line started
at $j$ at time $0$
has not reached $n$ at time $t$, that is $\alpha_{j}^n >t$.
For\vspace*{1pt} the reverse inclusion, if
$\alpha_{j}^n >t$, then the $n$-coalescent has more than $j$
blocks at time $t$, that is time $0$ in the lookdown model, and $\tau
_{j}^n >t$.
\end{pf}

Either the increasing sequence of the expected depths of the $n$-coalescent
%
\begin{equation}
\label{CDI}
\bigl(\mathbb{E}\bigl(\tau_1^n\bigr), n
\geq1\bigr)
\end{equation}
is bounded, or it goes to $\infty$.
In the first case, the coalescent is said to \textit{come down from
infinity}, whereas in the second case, it is said to \textit{stay infinite}.
These two classes of coalescent enjoy the following properties:
A coalescent that comes down from infinity does so immediately: not
only the increasing sequence of the depths stays bounded,
$\lim_{n \rightarrow\infty} \tau_1^n < \infty$ a.s., but also the
number of
blocks at each positive time remains bounded,
that is, $\lim_{n \rightarrow\infty} X^n(t) < \infty$ a.s. for $t>0$
(recall we work under the assumption $\Lambda\{1\}=0$).
For a coalescent which stays infinite, however, the increasing sequence
of the depths diverges, $\lim_{n \rightarrow\infty} \tau_1^n =
\infty$ a.s.,
and also the number of blocks at each nonnegative time diverges,
$\lim_{n \rightarrow\infty} X^n(t) = \infty$ a.s. for $t \ge0$. We
refer to
Pitman \cite{PI99} for the proof of these facts.

We take the opportunity to point out that, when the increasing sequence
$\tau_1^n$ has bounded moments, it automatically also has small uniform
exponential moments. The following simple argument may be found in
Limic \cite{L12}.
When the coalescent comes down from infinity, the depth $\tau_1:= \lim_{n \rightarrow\infty} \tau_1^n$ of the infinite coalescent
satisfies $p:=\mathbb{P}
(\tau_1 < 1 )>0$. Using the natural coupling of the $n$-coalescents for
the first inequality, and a finite induction on $t$ (that uses the same
coupling) for the second inequality, we obtain
\[
\mathbb{P}\bigl(\tau_1^n \ge t\bigr) \leq\mathbb{P}(
\tau_1 \ge t ) \leq (1-p)^t \qquad \mbox{for } t \in
\mathbb{Z}^+.
\]


From Lemma~\ref{tau-alpha} and the discussion above, the coalescent
comes down from infinity when the increasing sequence $(\alpha_1^n, n
\geq1)$ is a.s. bounded, that is, when the fixation line reaches
$\infty$ in finite
time a.s. Also, the coalescent stays infinite when
the increasing sequence $(\alpha_1^n, n \geq1)$ goes to $\infty$ a.s.,
that is, when the fixation line remains finite for all finite time a.s.

Last, we mention there is a simple criterion due to Schweinsberg \cite{SC01} that involves the probability measure $\Lambda(dx)$, for
discriminating between the two alternatives (coming down from infinity,
or staying infinite). There are examples of coalescent in both classes.

\begin{rem}
It would be interesting to find an analogous criterion on $\Lambda(dx)$
to discriminate between converging and diverging sequences
\[
\bigl(\operatorname{Var}\bigl(\tau_1^n\bigr), n \geq1\bigr).
\]
A first look at the Bolthausen--Sznitman coalescent suggests this
criterion should be distinct from the Schweinsberg criterion.
\end{rem}

\subsection{The transition rates}

Our next task is to determine the transition rates of the fixation line.

\begin{lem}
\label{propgamma}
For $1 \leq i<j$, the rate $\tilde{\Gamma}_{i,j}$ at which a fixation
line $(L(t), t \geq0)$ goes from $i$ to $j$ is
%
\begin{equation}
\label{defGammaij}
\hspace*{16pt}\quad\tilde{\Gamma}_{i,j}= \pmatrix{j \cr j-i+1} \int
_{[0,1]} \Lambda(dx) x^{-2} x^{j-i+1}
(1-x)^{i},\qquad  1 \leq i < j < \infty.
\end{equation}
\end{lem}

\begin{pf}
A fixation line jumps from $i$ to $j$ when, at the time of a
reproduction event,
$j-i+1$ levels exactly are chosen among the levels $1, 2, \ldots, j$, and
the level $j+1$ is not chosen.
For\vspace*{1pt} a reproduction event with asymptotic frequency $x$,
this has probability
$x^{j-i+1}   (1-x)^{i}$ for any unordered set with $j-i+1$ elements in
$\{1,\ldots, j\}$.
Counting the number of such sets, and integrating
with respect to the ``law'' of the asymptotic frequency $x$ gives the formula.
\end{pf}

The quantity $\tilde{\Gamma}_{i,j}$ should be compared with the rate
$\Lambda_{j,i}$ at which the block counting process
of the $n$-coalescent $(X^{n}(t), t \geq0)$ jumps from $j$ to $i$:
%
\begin{eqnarray}
\label{deflambdaij}
\Lambda_{j,i} &=&  \pmatrix{ j \cr j-i+1 } \int
_{[0,1]} \Lambda(dx) x^{-2} x^{j-i+1}
(1-x)^{i-1},
\nonumber
\\[-8pt]
\\[-8pt]
\eqntext{1 \leq i < j < \infty.}
\end{eqnarray}
Unlike the transitions of the block counting process, which involve (a
mixture of) binomial distributions, the transitions
of the fixation line therefore involve (a~mixture of) negative binomial
distributions.
The two quantities $\tilde{\Gamma}_{i,j}$ and $\Lambda_{j,i}$ differ in
general, still we have the following relationship.

\begin{lem}
\label{lemcombi}
For $i < j$, the rate $\tilde{\Gamma}_{i, \geq j}$ at which a fixation
line jumps from $i$ to a level ${\geq}j$ is
equal to the rate $\Lambda_{j,\leq i}$ at which the block counting
process jumps from $j$ to ${\leq}i$ blocks:
%
\begin{equation}
\label{hypergeom2}
\tilde{\Gamma}_{i, \geq j} = \Lambda_{j,\leq i}.
\end{equation}
\end{lem}

A computational proof is given in the \hyperref[app]{Appendix}.
For another instance of such a duality relationship, formulated in the
framework of measure valued process, we refer the reader to Lemma~5,
page 282 of Bertoin and Le Gall \cite{BLG03}.

The claim (\ref{hypergeom2}) may also be justified directly as follows:
a fixation line jumps from $i$ to a level ${\geq}j$ when, at the time
of a
reproduction event, at least $j-i+1$ levels are chosen among the levels
$1, 2, \ldots, j$, without any condition on
the level $j+1$. The same event backward corresponds to a coalescence
from $j$ blocks to ${\leq}i$ blocks.

Setting $j=i+1$ in (\ref{hypergeom2}), we obtain that the total rate at
which the fixation line jumps up from $i$
is equal to the total rate at which the block counting process jumps
down from $i+1$:
%
\begin{equation}
\label{hypergeom3}
\tilde{\Gamma}_{i, \geq i+1} = \Lambda_{i+1,\leq i},
\end{equation}
two quantities that we simply denote by $\Lambda_{i+1}$ in the following.

\subsection{The \texorpdfstring{$\Beta(2-\alpha,\alpha)$}{Beta(2-alpha,alpha)} family}

The Beta($2-\alpha,\alpha$) family of probability measures is given,
for $0< \alpha<2$, by
%
\begin{eqnarray}
\Lambda(dx) &=&  \Beta(2-\alpha,\alpha) (dx)
\nonumber
\\[-8pt]
\label{betafamily}
\\[-8pt]
\nonumber
&=&
\frac{1}{\Gamma
(2-\alpha
) \Gamma(\alpha)} x^{1-\alpha} (1-x)^{\alpha-1} \mathbf{1}_{[0,1]}(x)\,dx.
\end{eqnarray}
The $\Beta(2-\alpha,\alpha)$-coalescent interpolates between the star
like coalescent, which corresponds to the limit case
$\alpha=0$, the Bolthausen--Sznitman coalescent, $\alpha=1$, and the
Kingman coalescent, which corresponds to the limit case
$\alpha=2$. Example~15 in Schweinsberg \cite{SC01}, or (\ref{eqdepth})
and (\ref{eqdepth2}) in this paper, ensure that
the $\Beta(2-\alpha,\alpha$)-coalescent
comes down from infinity [see around (\ref{CDI})], if and only if
$\alpha>1$.

\begin{lem}
\label{lemfactorize}
When $\Lambda(dx)$ is given by \emph{(\ref{betafamily})} for some
$\alpha\in(0,2)$,
the jump rates $\tilde{\Gamma}_{i,i+j}$ of the fixation line $(L(t), t
\geq0)$ factorize as follows:
%
\begin{equation}
\label{factorize}
\tilde{\Gamma}_{i,i+j} = \frac{1}{\alpha  \Gamma(\alpha)}
\frac{\Gamma(i+\alpha)}{\Gamma(i)} \frac{\alpha}{ \Gamma(2-\alpha)} \frac{\Gamma(j-\alpha
+1)}{\Gamma
(j+2)}.
\end{equation}
\end{lem}

Conversely, it is not difficult to show that the $\Beta(2-\alpha,\alpha)$ family contains all the probability
measures $\Lambda(dx)$ for
which $\tilde{\Gamma}_{i,i+j}$ factorizes as a product of a function of
$i$ and a function of $j$.
We stress that the transition rates $\Lambda_{j,j-i}$ of the block
counting process of the
$\Beta(2-\alpha,\alpha)$-coalescent
do not enjoy such a factorization property.

To sum up, adopting a backward viewpoint results in a seemingly
anecdotic change of the exponent of $(1-x)$ in the rate (\ref{defGammaij}) with respect to the rate (\ref{deflambdaij}),
which in turn yields a factorization for the $\Beta(2-\alpha,\alpha)$-coalescent.
This factorization will be the key to exact computations.

\begin{pf*}{Proof of Lemma~\protect\ref{lemfactorize}}
The claim follows from the following elementary calculation
\begin{eqnarray*}
\tilde{\Gamma}_{i,i+j} & =& \frac{1}{\Gamma(\alpha)\Gamma(2-\alpha
)} \pmatrix{i+j \cr j+1}
\int_{(0,1)} dx x^{-\alpha-1} (1-x)^{\alpha-1}
x^{j+1} (1-x)^{i} \,dx
\\
& =&  \frac{1}{\Gamma(\alpha)\Gamma(2-\alpha)} \frac
{(i+j)!}{(j+1)!(i-1)!} \Beta (j-\alpha+1,i+\alpha)
\\
& =& \frac{1}{\Gamma(\alpha)\Gamma(2-\alpha)} \frac{\Gamma
(i+\alpha
)}{\Gamma(i)} \frac{\Gamma(j-\alpha+1)}{\Gamma(j+2)}.
\end{eqnarray*}
\upqed\end{pf*}

Let $\mathcal{S}_j= \{ L_j(t),   t \geq0\}$ be the range of the
fixation line
started at $j$.
Lemma~\ref{lemfactorize} entails that the law of the translated range
$\mathcal{S}_j-j=\{ L_j(t)-j,   t \geq0\}$ does not depend on $j$ in the
Beta($2-\alpha,\alpha$) case.
We shall simply use $\mathcal{S}$
to denote this random set. The set $\mathcal{S}$ is the range of a renewal
process, and we compute its renewal measure.
We set
%
\begin{equation}
\label{eqpngenerayting}
\varphi_{\eta^{\star}}(s)= \cases{
\ds - s/\bigl((1-s) \log(1-s)\bigr),  & \quad $\mbox{if } \alpha= 1$,
\vspace*{3pt}\cr
\ds - (\alpha-1) s/ \bigl[ (1-s)^{\alpha} -(1-s) \bigr], & \quad$\mbox{if } \alpha
\in(0,2) \setminus\{1\}$.}
\end{equation}
%

\begin{prop}
When $\Lambda(dx)$ belongs to the $\Beta(2-\alpha,\alpha)$ family
given by (\ref{betafamily})
for some $\alpha\in(0,2)$,
the generating function of the renewal measure is
%
\begin{equation}
\label{eqpngenerayting0}
\sum_{i \geq0} \mathbb{P}(i \in
\mathcal{S}) s^{i}= \varphi _{\eta^{\star}}(s).
\end{equation}
\end{prop}

\begin{pf}
The random set $\mathcal{S}$ is a renewal point process on $\mathbb{Z}^+$ based on the
interarrival measure
%
\begin{equation}
\label{eqetabeta}
\eta \{j\} =\frac{\alpha}{ \Gamma(2-\alpha)}\frac{\Gamma
(j-\alpha+1)}{\Gamma(j+2)}, \qquad j \geq1.
\end{equation}
The measure $\eta$ is a probability measure, as confirmed by setting $s=1$
in the following computation of the generating function $\varphi_{\eta}(s)$ of $\eta$.
We first do the computation for $\alpha\in(0,2) \setminus\{1\}$:
\[
\varphi_{\eta}(s) = \sum_{j \geq1} \eta\{j\}
s^j = \sum_{j \geq1} \frac{-1}{1-\alpha} \pmatrix{j-\alpha \cr j+1} s^j = \frac{-1}{(1-\alpha) s} \sum
_{j \geq2} \pmatrix{\alpha\cr j} (-s)^{j}.
\]
Using the binomial theorem, we deduce that
\[
\varphi_{\eta}(s) = \frac{-1}{(1-\alpha)s } \bigl[ (1-s)^{\alpha
} - 1
+ \alpha s \bigr] = 1 + \frac{1}{(\alpha-1)s } \bigl[ (1-s)^{\alpha} - (1 - s)
\bigr].
\]
We now consider the case $\alpha=1$:
\[
\varphi_{\eta}(s)= \sum_{j \geq1} \eta\{j\}
s^j = \frac{1}{s} \sum_{j \geq1}
\frac{1}{j(j+1)} s^{j+1} = 1+ \frac{(1-s)\log(1-s)}{s},
\]
using for the last equality that the primitive of $s \mapsto-\log
(1-s)$ null at $0$ is
\[
s \mapsto(1-s)\log(1-s)+s.
\]
We deduce the
generating function $\varphi_{\eta^{\star}}(s)$
of the renewal measure using the renewal property. Let $\mathcal{S}=\{
0=L^0<L^1<L^2< \cdots\}$
be the enumeration of the elements of $\mathcal{S}$ in increasing
order. We have
\[
\varphi_{\eta^{\star}}(s) =\sum_{i \in\mathbb Z^+}
s^i \mathbb {P}(i \in\mathcal{S}) = \mathbb{E} \biggl(\sum
_{i \in\mathbb Z^+} s^{L^i} \biggr) = 1+ \mathbb{E}
\bigl(s^{L^1} \bigr) \mathbb{E} \biggl( \sum
_{i \in\mathbb{Z}^+} s^{L^i} \biggr),
\]
that is,
\[
\varphi_{\eta^{\star}}(s) = 1+ \varphi_{\eta}(s) \varphi_{\eta
^{\star}}(s),
\]
and the claim follows.
\end{pf}

\begin{rem}
The distribution $\eta$ in (\ref{eqetabeta}) is, up to a shift by $1$,
the offspring distribution in the reduced tree associated with an
$\alpha$-stable branching process; see Theorem~3.3.3(i) in \cite{DLG02}.
This points to the connection with $\alpha$-stable trees
mentioned at the end of the \hyperref[secfixation]{Introduction}.
\end{rem}

In two particular cases, the renewal measure $\mathbb{P}(i \in
\mathcal{S})$ is
explicit: in the case $\alpha=1/2$, we have
%
\begin{eqnarray}
\sum_{j \geq0} \mathbb{P}(j \in
\mathcal{S}) s^j &=& \frac{1}{2} \biggl(\frac{1}{\sqrt
{1-s}} + 1
\biggr)
\nonumber
\\[-8pt]
\label{eqalphaonehalf}\\[-8pt]
\nonumber
&=& \frac{1}{2} \sum_{j \geq0} \biggl(
\frac{\Gamma(j+1/2)}{\Gamma(1/2)   \Gamma(j+1)}+\mathbf{1}_{\{j=0\}
} \biggr) s^j
\end{eqnarray}
and in the case $\alpha=3/2$, we have
%
\begin{eqnarray}
\sum_{j \geq0} \mathbb{P}(j \in
\mathcal{S}) s^j &=&  \frac{1}{2} \biggl(\frac{1}{\sqrt{1-s}} +
\frac{1}{1-s} \biggr)
\nonumber
\\[-8pt]
\label{eqalphathreehalf}
\\[-8pt]
\nonumber
&=&  \frac{1}{2} \sum_{j \geq0}
\biggl(\frac{\Gamma(j+1/2)}{\Gamma(1/2)   \Gamma(j+1)}+1 \biggr) s^j.
\end{eqnarray}

The measure $\eta$ given by (\ref{eqetabeta}) is a probability measure,
therefore, we have,
from the definition of $\Lambda_{i+1}$
[short after (\ref{hypergeom3})] and (\ref{factorize}), that
%
\begin{equation}
\label{eqsim-nalpha}
\Lambda_{i+1}= \tilde{\Gamma}_{i,\geq i+1}=
\frac{1}{\alpha\Gamma
(\alpha)} \frac{\Gamma(i+\alpha)}{\Gamma(i)} \sim\frac{1}{\alpha\Gamma(\alpha)} i^{\alpha} \qquad
\mbox{as } i \rightarrow\infty,
\end{equation}
where $a_n \sim b_n$ means that $\lim_{n \rightarrow\infty} a_n/b_n =1$.
We also notice, for future use, that the transition rate from $i$
blocks to $1$ block satisfies
%
\begin{equation}
\label{eqlambdaj1}
\Lambda_{i,1}= \frac{1}{\Gamma(2-\alpha)} \frac{\Gamma
(i-\alpha
)}{\Gamma(i)}
\sim\frac{1}{\Gamma(2-\alpha)} i^{-\alpha} \qquad \mbox{as } i \rightarrow\infty.
\end{equation}

\section{Applications}
\label{sec2}

\subsection{The fixation line and the set of records}
\label{sec21}

The first application of the fixation line consists in the computation
of the probability for an\vspace*{1.5pt} integer $i$ to be a record.
Recall that the set $\mathcal{T}$
of records is the set
$\{i \geq2, \tau_{1}^{i} >\tau_{1}^{i-1}\}$ where the sequence $\tau_1^i$
is defined in the natural coupling of the $n$-coalescents; see the
\hyperref[secfixation]{Introduction} for the definition of this coupling.
Recall also that $\mathcal{S}_1= \{L_1(t), t \geq0\}$ stands for the
range of
the fixation line started at $1$.
We stress the proposition is valid for a general probability measure
$\Lambda(dx)$.

\begin{prop}
\label{lemrecord}
The marginal distribution of the set of records $\mathcal{T}$ satisfies
\[
\mathbb{P}(i \in\mathcal{T})= \frac{\mathbb{P}(i-1 \in\mathcal{S}_1)}{\Lambda_{i}},\qquad i \geq2.
\]
\end{prop}

\begin{pf}
If $\mathbf{e}$ denotes an exponential random variable with parameter
$1$ that
is independent of $\tau^{i-1}_1$ and $\{i \in\mathcal{T}\}$, it holds
%
\begin{equation}
\label{eqrecord} \tau^{i}_1=\tau^{i-1}_1
+ \mathbf{1}_{i \in\mathcal{T}} \mathbf{e},\qquad i \geq2,
\end{equation}
%
and we deduce
%
\begin{equation}
\label{eqrecord2}
\hspace*{12pt}\mathbb{P}(i \in\mathcal{T}) = \mathbb{E}\bigl(
\tau^{i}_1\bigr)-\mathbb {E}\bigl(\tau^{i-1}_1
\bigr) =\mathbb{E}\bigl(\alpha^{i}_1\bigr)-\mathbb{E}
\bigl(\alpha^{i-1}_1\bigr)= \mathbb{P}(i-1 \in
\mathcal{S}_1)/\Lambda_{i}
\end{equation}
using Lemma~\ref{tau-alpha} for the second equality, and relation
(\ref{hypergeom3}) for the third equality.
\end{pf}

Let $(\mathbf{e}_i)_{2 \leq i \leq n}$ be a collection of independent
exponential random variables with parameter $1$, also independent of
$\mathcal{T}$. Iterating (\ref{eqrecord}) yields
\[
\tau_1^n= \sum_{2 \leq i \leq n}
\mathbf{1}_{\{i \in\mathcal{T}\}} \mathbf{e}_i.
\]
%
Combining with
the discussion on coalescents which come down from infinity that
follows Lemma~\ref{eqdeftaualpha},
we deduce that the cardinality of the set $\mathcal{T}$ is a.s.
infinite or a.s.
finite. It is infinite when the coalescent stays infinite, and finite
when the coalescent comes down from infinity.

Recall that $\mathcal S \stackrel{(d)}{=} \mathcal S_1-1$ stands for the
shifted range of the fixation line. Using Proposition~\ref{lemrecord}
and formulas
(\ref{eqalphaonehalf}) and (\ref{eqsim-nalpha}), we obtain the
following expression for the record probabilities in the case $\alpha=1/2$:
\[
\mathbb{P}(i \in\mathcal{T}) = \frac{1}{2} \biggl( \frac{1}{2i-3} +
\mathbf{1}_{\{i=2\}} \biggr),\qquad  i \geq2.
\]
This result gains a clear interpretation in the representation of the
$\Beta(3/2,\break 1/2)$-coalescent found by
Abraham and Delmas \cite{AD13},
which uses the pruning at nodes of a labelled binary tree with $n$ leaves.
In case $\alpha=3/2$, we use (\ref{eqalphathreehalf}) instead of~(\ref{eqalphaonehalf}) to obtain
\[
\mathbb{P}(i \in\mathcal{T}) = \frac{3}{2} \frac{1}{(2i-1)(2i-3)}+
\frac{3}{4} \frac{ \Gamma(3/2) \Gamma(i-1)}{\Gamma(i+1/2)},\qquad  i \geq2.
\]
For general $\alpha\in(0,2)$, we compute the generating function of
the record probabilities.

\begin{prop}
\label{lemrecord2}
The marginal distribution of the set of records in the
$\Beta(2-\alpha,\alpha)$-coalescent has the following generating function:
\[
\sum_{i \geq2} \mathbb{P}(i \in\mathcal{T})
s^i = s^3 \int_{(0,1)} dx
\frac{- x}{(1-sx) \log(1-sx)}
\]
in the Bolthausen--Sznitman case $\alpha=1$, and
\[
\sum_{i \geq2} \mathbb{P}(i \in\mathcal{T})
s^i= \alpha (1-\alpha) s^3 \int_{(0,1)}
dx \frac{x}{
(1-x)^{1-\alpha}  [ (1-sx)^{\alpha} -(1-sx) ]}
\]
in case $\alpha\in(1,2) \setminus\{1\}$.
\end{prop}

\begin{pf}
We do the following computation:
\begin{eqnarray*}
\sum_{i \geq2} \mathbb{P}(i \in\mathcal{T})
s^i 
&=& \sum_{i \geq2} \alpha
\frac{\Gamma(\alpha)   \Gamma
(i-1)}{\Gamma
(i-1+\alpha)} \mathbb{P}(i-2 \in\mathcal{S}) s^i
\\
&=&\int_{(0,1)} dx \alpha (1-x)^{\alpha-1} s^2
\sum_{i
\geq
2} \mathbb{P}(i-2 \in\mathcal{S})
(sx)^{i-2}
\\
&=& \int_{(0,1)} dx \alpha (1-x)^{\alpha-1} s^2
\varphi _{\eta
^\star}(sx)
\end{eqnarray*}
using Proposition~\ref{lemrecord}, the definition of
$\mathcal{S}=\mathcal{S}_1-1$ and
formula (\ref{eqsim-nalpha}) at the first equality,
the link between Gamma and Beta functions at the second equality as
well as the Fubini--Tonelli theorem.
The claim now follows substituting $\varphi_{\eta^\star}$ by its
value given
in (\ref{eqpngenerayting}), distinguishing whether $\alpha\in(0,2)
\setminus\{1\}$ or $\alpha=1$.
\end{pf}

\begin{cor}
\label{cordepth}
The depth $\tau_{1}^n$ of the $n$-$\Beta(2-\alpha,\alpha)$-coalescent almost surely converges as $n \rightarrow\infty$
in the natural coupling
to a random variable $\tau_1$ with expectation:
%
\begin{equation}
\label{eqdepth}
\mathbb{E}(\tau_1)= \alpha(\alpha-1) \int
_{(0,1)} dx \frac{x}{(1-x)^{2-\alpha}
[1-(1-x)^{\alpha-1}]}
\end{equation}
in case $\alpha\in(1,2)$.
\end{cor}

\begin{pf}
The sequence $\tau_{1}^n$ is increasing in the natural coupling of the
$n$-coalescents
whence the a.s. convergence. For the expectation: set $s=1$ in
Proposition~\ref{lemrecord2}, and use the first equality in
(\ref{eqrecord2}): this gives a telescopic sum
with sum $\mathbb{E}(\tau_1)$.
\end{pf}

Since $\mathbb{E}(\tau_1)$ is finite according to (\ref{eqdepth}), the
$\Beta(2-\alpha,\alpha)$-coalescent with $\alpha\in(1,2)$ comes down
from infinity.
On the other hand, when $\alpha\in(0,1]$, we have, using again
Proposition~\ref{lemrecord2} with $s=1$ that
%
\begin{equation}
\label{eqdepth2}
\lim_{n \rightarrow\infty} \mathbb{E}\bigl(
\tau_1^n\bigr) = \infty,
\end{equation}
and the $\Beta(2-\alpha,\alpha)$-coalescent with $\alpha\in(0,1]$
therefore stays infinite.
In this case, the suitably rescaled random variables $\tau_1^n$, as $n
\rightarrow\infty$,
have been proved to converge in law. We refer to \cite{GN14} and the
references therein
for a refined study of these random variables.

\subsection{The fixation line and the last coalescence}
\label{sec22}

We consider the block counting process
$(X^n(t), t \geq0)$ of the $n$-coalescent.
Its embedded Markov chain starts at $n$, and has transitions probabilities:
%
\begin{equation}
\label{eqPji}
P_{ji}=\frac{\Lambda_{j,i}}{\Lambda_j},\qquad  j > i \geq1,
\end{equation}
with $\Lambda_{j,i}$ the transition rate from $j$ to $i$ blocks defined
in (\ref{deflambdaij}) and $\Lambda_j= \sum_{i<j} \Lambda_{ji}$.
In this subsection, we consider the problem of the convergence of the
time-reversal of this Markov chain as
the initial number of blocks $n$ goes to $\infty$. Unlike the
transition probabilities of the original chain,
the transition probabilities of the chain reversed in time
depend on the starting point $n$, and we shall\vspace*{1pt} use $\tilde{P}_{ij}^n$
to denote the
transition probability of the reversed chain from $i$ to $j$, $1 \leq i
< j$,
when the original chain is starting at $n$.
For each integer $i$, we have a collection, indexed by the integer $n$,
of probability measures
\[
\bigl(\tilde{P}_{ij}^n, i \leq j \leq n\bigr),
\]
and we propose to study the weak convergence of this family of
probability measures as $n \rightarrow\infty$.
We first observe that the question for an arbitrary $i\geq1$ may be
reduced to the
case $i=1$:
if $\mathcal{R}^n= \{X^n(t), t \geq0 \}$ stands for the range of the block
counting process of the $n$-coalescent,
we have the equality
%
\begin{equation}
\label{eqpij-n}
\mathbb{P}\bigl(i \in\mathcal{R}^n\bigr)
\tilde{P}^n_{ij} = \mathbb{P}\bigl(j \in
\mathcal{R}^n\bigr) P_{ji},
\end{equation}
that is a particular instance of Nagasawa's formula (cf., e.g., Chapter
III.42, III.46 in \cite{RW00}).
We form two further equations that we call the second and the third
equations, by specializing the first equation (\ref{eqpij-n}) to the
couple $(1,j)$ and $(1,i)$, respectively. We then divide the first
equation by the second, and then multiply by the third equation, the
operations being term by term. This gives
%
\begin{equation}
\label{eqPijn}
\tilde{P}^n_{ij} = 
\frac{ P_{ji}   P_{i1}}{P_{j1}} \frac{\tilde{P}^n_{1j}}{\tilde
{P}^n_{1i}}.
\end{equation}
The following proposition gives an expression of the distribution
$(\tilde{P}^n_{1j}, 1<j \leq n)$ in term of quantities related to the
fixation line.
This is the essential conceptual step in the study of the last
coalescence since the next steps, carried out in the case
of the Beta($2-\alpha,\alpha$)-coalescent in the next subsection,
reduce to the computation of $\mathbb{E}(\alpha_{j}^n)$.

\begin{prop}
\label{proplastco}
The distribution $(\tilde{P}^n_{1j}, 2 \leq j \leq n)$ of the number of
blocks involved in the last coalescence of a $n$-coalescent
satisfies
%
\begin{equation}
\label{eqP1jn} \tilde{P}^n_{1j} = \Lambda_{j1}
\bigl[\mathbb{E}\bigl(\alpha _{j-1}^n\bigr)-\mathbb{E}\bigl(
\alpha _{j}^n\bigr) \bigr],\qquad  2 \leq j \leq n.
\end{equation}
\end{prop}

\begin{pf}
We compute
\[
\tilde{P}^n_{1j}=P_{j1} \mathbb{P}\bigl(j \in
\mathcal{R}^n\bigr) =\Lambda_{j1} \mathbb{E} \biggl( \int
_{0}^{\infty} ds \mathbf {1}_{\{X^{n}(s)=j\}
} \biggr)
=\Lambda_{j1} \bigl[\mathbb{E}\bigl( \tau^{n}_{j-1}
\bigr)-\mathbb{E}\bigl( \tau ^{n}_{j} \bigr) \bigr],
\]
setting $i=1$ in (\ref{eqpij-n}) for the first equality,
using the definition (\ref{eqPji}) of $P_{j1}$ and the fact that the
block counting process spends an exponential time
at $j$ with parameter $\Lambda_j$ when $j \in\mathcal{R}^n$ for the second
equality, and the pathwise relation
$\tau^{n}_{j-1}=\tau^{n}_{j} + \int_{(0,\infty)} ds   \mathbf
{1}_{\{
X^{n}(s)=j\}}$
for the last equality. We conclude using Lemma \ref{tau-alpha}.
\end{pf}

For $j \geq1$, and with $\alpha_j$ the increasing limit of
$(\alpha_j^n, n \geq j)$, we have\break $\lim_{n \rightarrow\infty}
\mathbb{E}(\alpha_j^n) =
\mathbb{E}(\alpha_j)$.
In case the coalescent comes down from infinity,
$\mathbb{E}(\alpha_j)<\infty$, and
\[
\lim_{n \rightarrow\infty} \mathbb{E}\bigl(\alpha_{j-1}^n
\bigr)-\mathbb {E}\bigl(\alpha_{j}^n\bigr) = \mathbb{E}(
\alpha _{j-1})-\mathbb{E}(\alpha_{j})
\]
for each $j \geq1$. The convergence of
$(\tilde{P}^n_{ij}, n \geq2)$ for arbitrary $i<j$
follows from
(\ref{eqPijn}) and (\ref{eqP1jn}).
A more interesting case is when
the coalescent stays infinite. Then we cannot directly conclude to the
convergence of
the difference of the expectations $\mathbb{E}(\alpha
_{j-1}^n)-\mathbb{E}(\alpha_{j}^n)$.
Setting $\mathcal{S}_j=\{L_j(t), t \geq0\}$ for the range of the
fixation line
$L_j$ started at $j$,
we have
%
\begin{equation}
\label{eqproblem}
\mathbb{E}\bigl(\alpha_{j-1}^n\bigr)-
\mathbb{E}\bigl(\alpha_{j}^n\bigr) =\sum
_{j-1 \leq i \leq n-1} \bigl[ \mathbb{P}(i \in\mathcal {S}_{j-1})-
\mathbb{P}(i \in\mathcal{S}_{j}) \bigr] \frac{1}{\Lambda_{i+1}}
\end{equation}
using that the rate at which the fixation line leaves $i$ is $\Lambda_{i+1}$.
If the coalescent stays infinite, the series with general term
$1/\Lambda_{i}$ is easily
seen to diverge, since
\[
\mathbb{E}\bigl(\tau_{1}^n\bigr) = \sum
_{2 \leq i \leq n} \frac{1}{\Lambda
_i} \mathbb{P}\bigl(i\in
\mathcal{R}^n\bigr) \leq\sum_{2 \leq i \leq n}
\frac{1}{\Lambda_{i}}.
\]
Proving convergence in (\ref{eqproblem}) as $n \rightarrow\infty$ therefore
requires a further study of
$ [ \mathbb{P}(i \in\mathcal{S}_{j-1})-\mathbb{P}(i \in
\mathcal{S}_j)]$, which is a
difficult issue in general.
The factorization property satisfied by the $\Beta(2-\alpha,\alpha)$-coalescent; see Lemma~\ref{lemfactorize}, allows to
circumvent the difficulty.

\subsubsection{The \texorpdfstring{$\Beta(2-\alpha,\alpha)$}{Beta(2-alpha,alpha)}-coalescent}
\label{secbetaalpha}

Before stating our main theorem, it is perhaps opportune to recall the
statement of the problem of the last coalescence in a self contained manner:
The block counting process $(X^{n}(t), t \geq0)$ is a Markov chain
started at $n$ and a.s. absorbed at $1$ in finite time.
In the case of the $\Beta(2-\alpha,\alpha)$-coalescent, the transitions
rates of the block counting process from $j$ to $i$ are given by
\[
\Lambda_{j,i} = \frac{\Gamma(j+1)}{\Gamma(j)} \frac{\Gamma(j-i+1-\alpha
)}{\Gamma
(j-i+2)}
\frac{\Gamma(i+ \alpha-1)}{\Gamma(i)},\qquad 1 \leq i < j < \infty.
\]
What is the law of the last jump of $(X^{n}(t), t \geq0)$ as $n
\rightarrow
\infty$? The following theorem answers the question.

\begin{theo}
\label{propconv-Delta}
The\vspace*{1pt} distribution $(\tilde{P}_{1j}^n, j\geq2)$ of the number of blocks
implied in the last coalescence of the $n$-$\Beta(2-\alpha,\alpha)$-coalescent
weakly converges as $n \rightarrow\infty$ toward a distribution
$(\tilde{P}_{1j}, j \geq2)$ with generating function
%
\begin{equation}
\label{eqgeneratingalphaeq1}
\sum_{j \geq2} \tilde{P}_{1j}
s^j = s \int_{(0,1)} dx \frac{\log(1-sx)}{\log(1-x)}
\end{equation}
in the Bolthausen--Sznitman case $\alpha=1$, and
%
\begin{equation}
\label{eqgeneratingalphaneq1}
\sum_{j \geq2} \tilde{P}_{1j}
s^j = \alpha s \int_{(0,1)} dx \frac{1}{1-(1-x)^{1-\alpha}}
\biggl[\frac{1}{(1-sx)^{1-\alpha}}-1 \biggr]
\end{equation}
in case $\alpha\in(0,2) \setminus\{1\}$.
\end{theo}

Setting $s=1$ in formulas (\ref{eqgeneratingalphaeq1}) and
(\ref{eqgeneratingalphaneq1})
allows to see that the
collection $(\tilde{P}_{1j}, j \geq2)$ is a probability measure.
Proposition~1.5 of Abraham and Delmas \cite{AD13+} contains the result
for $\alpha\in(0,1/2]$.
The proof given there relies on a connection with the pruning of L\'evy trees.
In the case $\alpha=1$, we deduce from (\ref{eqgeneratingalphaeq1}) that
%
\begin{eqnarray}
\tilde{P}_{1j} &=&  \frac{1}{j-1} \int_{(0,1)} dx
\frac
{x^{j-1}}{-\log
(1-x)}
\nonumber
\\
&= &\frac{1}{j-1} \int_{(0,\infty)} \frac{du}{u}
\bigl(1-\mathrm{e}^{-u}\bigr)^{j-1} \mathrm{e}^{-u}
\nonumber
\\
\label{eqGM}
&=& \frac{1}{j-1} \int_{(0,\infty)} \frac{du}{u} \sum
_{0 \leq k
\leq
j-1}\pmatrix{j-1 \cr k} (-1)^k
\bigl(\mathrm{e}^{-(k+1)u}-\mathrm{e}^{-u}\bigr)
\\
&=& \frac{1}{j-1} \sum_{1 \leq k \leq j-1} \pmatrix{j-1 \cr k}
(-1)^k \int_{(0,\infty)} \frac{du}{u}
\bigl(\mathrm{e}^{-(k+1)u}-\mathrm{e}^{-u}\bigr)
\nonumber
\\
&=&  \frac{1}{j-1} \sum_{1 \leq k \leq j-1} \pmatrix{j-1 \cr k}
(-1)^{k+1} \log(k+1),\nonumber
\end{eqnarray}
using\vspace*{1pt} the change of variable $x = 1-\mathrm{e}^{-u}$ at the second equality,
expanding $(1-\mathrm{e}^{-u})^{j-1}$ with the binomial theorem, and
compensating the resulting $\mathrm{e}^{-(k+1)u}$ by $\mathrm{e}^{-u}$ at the third
equality (this has no effect), so that each of the $j$ terms in the sum
are integrable functions of the variable $u$. (This cancels the term
$k=0$.) Then we consider each of the $j-1$ integrals separately at the
fourth equality, and each integral assumes the form of a Frullani
integral (cf.  \cite{O49},  e.g.):
\[
\int_{(0,\infty)} \frac{du}{u} \bigl(f(au)-f(bu)\bigr)\qquad
\mbox{with } f(u)= \mathrm{e}^{-u}, a=(k+1) \mbox{ and } b=1.
\]
A direct calculation ensures this integral is equal to $(f(0)-f(+\infty
)) \log(b/a)$, where $f(0)$ and $f(+\infty)$ are the limits of $f$ at
$0$ and $+\infty$, respectively, and
this gives the expression (\ref{eqGM}). This expression of $\tilde
P_{1j}$ is due to Goldschmidt and Martin \cite{GM05}, who obtained it
using a connection with the pruning of recursive trees.
It is interesting to observe the diversity of the methods at work in
\cite{AD13+,GM05} and the present paper.

\begin{pf*}{Proof of Theorem~\protect\ref{propconv-Delta}}
Recall $\mathcal{S}$ denotes the range of the renewal point process on
$\mathbb{Z}^+$
containing $0$ and with interarrival times with law
$\eta$ given by (\ref{eqetabeta}). We compute
%
\begin{eqnarray}
\mathbb{E}\bigl(\alpha_{j-1}^n\bigr)-\mathbb{E}\bigl(
\alpha_{j}^n\bigr) &= & \sum_{j-1 \leq i \leq n-1}
\mathbb{P}(i \in\mathcal{S}_{j-1}) \frac{1}{\Lambda
_{i+1}} - \sum
_{j \leq i \leq n-1} \mathbb{P}(i \in\mathcal {S}_{j})
\frac
{1}{\Lambda_{i+1}}
\nonumber
\\
\label{eqsum}&= & \sum_{j \leq i \leq n} \mathbb{P}(i-j \in\mathcal{S})
\frac{1}{\Lambda
_{i}} - \sum_{j \leq i \leq n-1} \mathbb{P}(i-j \in
\mathcal{S}) \frac{1}{\Lambda
_{i+1}}
\\
\nonumber
&= & \sum_{j \leq i \leq n-1} \mathbb{P}(i-j \in
\mathcal{S}) \biggl(\frac{1}{\Lambda
_{i}}-\frac{1}{\Lambda_{i+1}} \biggr) + \mathbb{P}(n-j
\in\mathcal{S}) \frac{1}{\Lambda_{n}},
\end{eqnarray}
beginning as in (\ref{eqproblem})
for the first equality, using the definition of the translated range
$\mathcal{S}=\mathcal{S}_j-j$,
that is independent of $j$ in the $\Beta(2-\alpha,\alpha)$ setting, and
then changing the index in the first sum
at the second equality.
Bounding $\mathbb{P}(i-j \in\mathcal{S})$ from above by $1$, and
using the positivity
of $\Lambda_n$, we obtain the following upper bound:
%
\begin{eqnarray}
\sum_{j \leq i \leq n-1} \mathbb{P}(i-j \in
\mathcal{S}) \biggl(\frac{1}{\Lambda
_{i}}-\frac{1}{\Lambda_{i+1}} \biggr) &\leq & \sum
_{ j \leq i \leq n-1} \frac{1}{\Lambda_{i}}-\frac
{1}{\Lambda_{i+1}}
\nonumber
\\[-8pt]
\label{eqbounded}
\\[-8pt]
\nonumber
& =&
\frac{1}{\Lambda_{j}} - \frac{1}{\Lambda_{n}} \leq\frac
{1}{\Lambda
_{j}}.
\end{eqnarray}
The series on the left-hand side of (\ref{eqbounded}) has nonnegative
terms [the sequence $(\Lambda_{i}, i \geq2)$ is nondecreasing] and is
bounded, therefore, it converges.
Also, the sequence $(\Lambda_{j}, j \geq2)$ goes to $\infty$ by
(\ref{eqsim-nalpha}).
The second term in (\ref{eqsum}) then goes to~$0$.
Using (\ref{eqP1jn}), we conclude that the limit as $n \rightarrow
\infty$ of
the quantities
$\tilde{P}_{1j}^n$ exists, we denote it by $\tilde{P}_{1j}$.
Setting $k=i-j$, we have
\[
\tilde{P}_{1j} = \Lambda_{j1} \sum
_{k \geq0} \mathbb{P}(k \in\mathcal{S}) \biggl(\frac
{1}{\Lambda_{k+j}}-
\frac
{1}{\Lambda_{k+j+1}} \biggr) < \infty. 
\]
Setting the explicit values (\ref{eqsim-nalpha})
and (\ref{eqlambdaj1}) of $\Lambda_{j}$ and $\Lambda_{j1}$
gives
\[
\tilde{P}_{1j} = \frac{1}{j-1} \sum_{k \geq0}
\mathbb{P}(k \in\mathcal{S}) \biggl[\frac{1}{k+j-1}- \frac{1}{k+j}
\biggr]
\]
in case $\alpha=1$, and
\[
\tilde{P}_{1j} =\frac{\alpha^2  \Gamma(\alpha)}{\Gamma(2-\alpha)} \frac{\Gamma(j-\alpha)}{\Gamma(j)} \sum
_{k \geq0} \mathbb{P}(k \in\mathcal{S}) \frac{\Gamma(k+j-1)}{\Gamma
(k+j+\alpha)}
\label{eqgeneratingalphaneq1-check2}
\]
in case $\alpha\in(0,2) \setminus\{1\}$.
Recall the expression (\ref{eqpngenerayting0}) for the generating\vspace*{0.5pt}
function of the numbers
$\mathbb{P}(k \in\mathcal{S})$.
Multiplying both sides of (\ref{eqpngenerayting0}) by
$s^{j-2}(1-s)^\alpha$, integrating with respect to $s \in(0,1)$ and
using Fubini--Tonelli theorem, we deduce
\[
\sum_{k \geq0} \mathbb{P}(k \in\mathcal{S}) \biggl[
\frac
{1}{k+j-1}- \frac{1}{k+j} \biggr] = - \int_{(0,1)}
ds \frac{ s^{j-1}}{ \log(1-s)}
\]
in case $\alpha=1$,
and
\[
\sum_{k \geq0} \mathbb{P}(k \in\mathcal{S})
\frac{\Gamma
(k+j-1)\Gamma(\alpha
+1)}{\Gamma(k+j+\alpha)} = -(\alpha-1) \int_{(0,1)} ds
\frac{s^{j-1}}{1-(1-s)^{1-\alpha}}
\]
in case $\alpha\in(0,2) \setminus\{1\}$, using also the expression
of the Beta function in term of the Gamma function.
From the last four equations displayed, we obtain
%
\begin{equation}
\label{eqalpha1}
\tilde{P}_{1j}= \frac{-1}{j-1}\int
_{(0,1)} dx \frac
{x^{j-1}}{\log(1-x)}
\end{equation}
in case $\alpha=1$, and
%
\begin{equation}
\label{eqgeneratingalphaneq1-check}
\tilde{P}_{1j} = (-1)^{j-1} \alpha \pmatrix{\alpha-1
\cr j-1} \int_{(0,1)} dx \frac
{x^{j-1}}{1-(1-x)^{1-\alpha}}
\end{equation}
in case $\alpha\in(0,2) \setminus\{1\}$,
which imply, respectively, (\ref{eqgeneratingalphaeq1}) and (\ref
{eqgeneratingalphaneq1}), multiplying by $s^j$
and summing over $j \geq2$.
\end{pf*}

\begin{cor}
\label{corhitting}
The probability for an integer $j \geq2$ to be in the range $\mathcal
{R}^n$ of the
block counting process of the $n$-$\Beta(2-\alpha,\alpha)$-coalescent converges as $n \rightarrow\infty$ and
\[
\lim_{n \rightarrow\infty} \mathbb{P}\bigl(j \in\mathcal{R}^n
\bigr) = (j-1) \int_{(0,1)} dx \frac
{x^{j-1}}{-\log(1-x)}
\]
in the Bolthausen--Sznitman case $\alpha=1$, and
%
\begin{equation}
\label{eqhitting}\hspace*{15pt} \lim_{n \rightarrow\infty} \mathbb{P}\bigl(j \in
\mathcal{R}^n\bigr) = \frac{1}{\Gamma(\alpha)} \frac{\Gamma(j-1+\alpha)}{\Gamma
(j-1)} \int
_{(0,1)} dx x^{j-1} \frac{1-\alpha}{1-(1-x)^{1-\alpha}}
\end{equation}
in case $\alpha\in(0,2) \setminus\{ 1 \}$.
\end{cor}

Notice the integrands in both integral representations are nonnegative
whatever the value of $\alpha\in(0,2)$.
Also, the Bolthausen--Sznitman case in Corollary~\ref{corhitting}
corresponds to the statement of Theorem~1.1 in M\"ohle
\cite{MO13}, and the case
$\alpha\in(0,2) \setminus\{1\}$
answers a question posted in the same paper; see Remark~3. The question has
also been answered independently by M\"ohle in \cite{MO13+} in a
subsequent paper.

\begin{pf*}{Proof of Corollary~\protect\ref{corhitting}}
This is a consequence of equation (\ref{eqpij-n}) together
with formula (\ref{eqalpha1}) in the Bolthausen--Sznitman case
$\alpha=1$,
and together with formula (\ref{eqgeneratingalphaneq1-check}) in the
case $\alpha\in(0,2) \setminus\{1\}$.
\end{pf*}

In case $\alpha\in(1,2)$, the coalescent comes down from infinity and
it is possible to consider directly the range $\mathcal{R}$ of the infinite
coalescent, defined as
the almost sure (local) limit of the
$\mathcal{R}^n$:  $\mathbf{1}_{i \in\mathcal{R}}= \lim_{n \rightarrow\infty}
\mathbf{1}_{i \in\mathcal{R}^n}$.
Dominated convergence theorem then ensures that the right-hand side of
(\ref{eqhitting}) corresponds
to $\mathbb{P}(j \in\mathcal{R})$.
We propose to write
\[
\mathbb{P}(j \in\mathcal{R}):=\lim_{n \rightarrow\infty} \mathbb {P}\bigl(j
\in\mathcal{R}^n\bigr)
\]
whatever the value of $\alpha\in(0,2)$: this is, however, an abuse of
notation since we do not give a meaning
to $\mathcal{R}$
when $\alpha\in(0,1]$.

\begin{cor}
\label{corhitasymptot}
The probability for an integer $j \geq2$ to be in the range of the
block counting process of the $\Beta(2-\alpha,\alpha)$-coalescent
satisfies
\[
\lim_{j \rightarrow\infty} \mathbb{P}(j \in\mathcal{R}) = \alpha-1,
\]
in case $\alpha\in(1,2)$, and
\[
\mathbb{P}(j \in\mathcal{R}) \sim\frac{1-\alpha}{\Gamma(\alpha
)} j^{\alpha-1} \qquad \mbox{as } j \rightarrow\infty,
\]
in case $\alpha\in(0,1)$.
\end{cor}

The asymptotics for $\alpha\in(1,2)$ have been previously derived in
Berestycki et al.
using a connection with $\alpha$-stable continuous tree;
see Theorem~1.8 of \cite{BBS08}. The Bolthausen--Sznitman case $\alpha
=1$ has been covered in M\"ohle \cite{MO13}, whose Corollary~1.2
states that
\[
\mathbb{P}(j \in\mathcal{R}) \sim\frac{1}{\log(j)} \qquad \mbox{as } j \rightarrow
\infty.
\]
\begin{pf*}{Proof of Corollary~\protect\ref{corhitasymptot}}
We first consider the case $\alpha\in(0,1)$. Estimating the left
factor in
(\ref{eqhitting}) is easy:
%
\begin{equation}
\label{eqGammasim}
\Gamma(j-1+\alpha) \sim j^\alpha {\Gamma(j-1)}\qquad \mbox{as } j \rightarrow \infty.
\end{equation}
For the remaining integral factor in (\ref{eqhitting}), we write
%
\begin{equation}
\label{eqsim1}
\int_{(0,1)} dx x^{j-1}
\frac{1}{1-(1-x)^{1-\alpha}}= \int_{(0,1)} dx x^{j-2} h(x)
\end{equation}
for $h(x)= x/  [ 1-(1-x)^{1-\alpha}  ]$.
Then we decompose the integral as follows:
\[
\int_{(0,1)} dx x^{j-2} h(x) = \frac{1}{j-1}
\biggl[ h(1) + \int_{(0,1)} dx (j-1) x^{j-2}
\bigl(h(x)-h(1)\bigr) \biggr].
\]
Fix $\varepsilon>0$. From the continuity of $h$ at $1$, there exists
$\eta
>0$ such that such $|h(x)-h(1)| \leq\varepsilon/2$
for $x \in(1-\eta,1]$, and
\[
\int_{(0,1)} dx (j-1) x^{j-2} \bigl|h(x)-h(1)\bigr| \leq
\frac{\varepsilon}{2} + 2 \|h\|_{\infty} (j-1) (1-\eta)^{j-2} \leq
\varepsilon
\]
for $j$ large enough. Therefore, the left-hand side
of (\ref{eqsim1}) is equivalent to $1/j$, and the claim follows in the
case $\alpha\in(0,1)$.
For the case $\alpha\in(1,2)$, it is more convenient to rewrite the
integral factor as follows:
\[
\int_{(0,1)} dx x^{j-1} \frac{1-\alpha}{1-(1-x)^{1-\alpha}}= (
\alpha-1) \int_{(0,1)} dx \frac{x^{j-1} (1-x)^{\alpha-1}
}{1-(1-x)^{\alpha-1}}.
\]
Then we write
\[
\int_{(0,1)} dx \frac{x^{j-1} (1-x)^{\alpha-1} }{1-(1-x)^{\alpha-1}}= \int_{(0,1)}
dx x^{j-2} (1-x)^{\alpha-1} h(x),
\]
for $h(x)=x/ [1-(1-x)^{\alpha-1}  ]$ this time. The same
reasoning as before allows to conclude that
\begin{eqnarray*}
\int_{(0,1)} dx x^{j-2} (1-x)^{\alpha-1} h(x)
&\sim &  h(1) \int_{(0,1)} dx x^{j-2} (1-x)^{\alpha-1}\\
&=&
\Gamma(\alpha) \frac{\Gamma(j-1)}{\Gamma(j-1+\alpha)},
\end{eqnarray*}
where the equivalent is taken as $j \rightarrow\infty$.
The last constant is the inverse of the first factor in
(\ref{eqhitting}), and the claim follows for $\alpha\in(1,2)$.
\end{pf*}

The definition of the block counting process of the $n$-coalescent entails
\[
\mathbb{E}\bigl(\tau_1^n\bigr) = \sum
_{2 \leq j \leq n} \frac{\mathbb{P}(j
\in\mathcal{R}^n)}{\Lambda_j}.
\]
Taking the $n \rightarrow\infty$ limit in this formula
gives an alternative proof of Corollary~\ref{cordepth} on the expected
depth of the $\Beta(2-\alpha,\alpha)$-coalescent
for $\alpha\in(1,2)$.

\subsection{Depth of the \texorpdfstring{$\Beta(1,1)$}{Beta(1,1)}-coalescent}
We propose to investigate further the Bolthausen--Sznitman coalescent
associated with $\Lambda(dx)= \mathbf{1}_{[0,1]} (x) \, dx$.
This coalescent stays infinite.
In fact, it plays a special r\^ole in the class of the $\Beta(2-\alpha,\alpha)$-coalescents, since
it separates those coalescents which come down from infinity, $\alpha>
1$, from those which stay infinite,
$\alpha\leq1$. Setting $\alpha=1$ in (\ref{factorize}) gives
\[
\tilde{\Gamma}_{i,i+j} = \frac{i}{j(j+1)}.
\]
Therefore, the fixation line $(L(t), t \geq0)$ is a continuous time
discrete state space
branching process (and this is the only coalescent for which this is
the case). We shall call this process the discrete Neveu branching
process, after \cite{N92}.
It has offspring distribution:
%
\begin{equation}
\label{eqmu-BS}
\mu\{j\} = \eta\{j-1\} = \frac{1}{j(j-1)}, \qquad j \geq2,
\end{equation}
since a jump of $j-1$ for the fixation line corresponds to the arrival
of $j$ children together with the death of
the father.
The generating function associated with the offspring distribution $\mu
$ is
\[
\varphi_{\mu}(s) = s \varphi_{\eta}(s) = s + (1-s) \log(1-s),
\qquad 0 \leq s < 1.
\]
The offspring distribution $\mu$ has infinite mean, but the branching
process is conservative, meaning it does not reach $+\infty$
in finite time. This agrees with our observation (after Lemma~\ref
{tau-alpha}) that the fixation line $(L(t), t \geq0)$
remains finite for coalescents which stay infinite, and the
Bolthausen--Sznitman coalescent stays infinite.
The rate of increase of $(L(t), t \geq0)$ is well known (see Grey
\cite{G77}, e.g.) we nevertheless include a proof for
the ease of reference.

\begin{proposition}\label{prop38}
In the Bolthausen--Sznitman case $\Lambda(dx)= \mathbf{1}_{[0,1]} (x)
\,dx$,
we have
%
\begin{equation}
\label{asymptoticBS}
\mathrm{e}^{-t} \log L_1(t) \rightarrow\mathbf{e}
\qquad \mbox{a.s. as } t \rightarrow\infty,
\end{equation}
with $\mathbf{e}$ an exponential random variable with parameter $1$.
\end{proposition}

This growth rate strongly contrasts with the exponential growth rate
satisfied by supercritical
branching processes with a finite mean offspring distribution; see the
Seneta--Heyde theorem.

\begin{pf*}{Proof of Proposition~\protect\ref{prop38} (after \cite{G77})}
We begin with general considerations on continuous-time branching processes.
The generating function $f_t(s)=\mathbb{E}(s^{L_1(t)})$ of $L_1(t)$
may be computed from the infinitesimal generating function $\phi(s)=
\varphi_{\mu}(s)-s$
using the partial differential equation
\[
\cases{\ds\partial_{t} f_t(s) =
\phi\bigl(f_t(s)\bigr),
\vspace*{3pt}\cr
\ds f_0(s)  = s,}
\]
see, for example, Harris \cite{H63}, Chapter V.
The function $f_t$ is a bijection from $[0,1]$ into itself,
with right-continuous inverse function $g_t$.
The process $(g_t(s)^{L_1(t)},  t \geq0)$ is Markov and has constant
expectation since
\[
\mathbb{E}\bigl(g_t(s)^{L_1(t)}\bigr)=f_t
\bigl(g_t(s)\bigr)=s.
\]
Therefore, it is a $[0,1]$-valued martingale that almost surely
converges towards a limiting random variable $V$
as $t \rightarrow\infty$.
At this point, we take advantage of the explicit formulas available in
our case:
\[
\phi(s)= \varphi_\mu(s)-s= (1-s) \log(1-s),
\]
which entails
\[
f_t(s)=1-(1-s)^{\mathit{e}^{-t}} \quad\mbox{and}\quad g_t(s)=1-(1-s)^{\mathit{e}^{t}}.
\]
We now compute, for $\alpha>0$,
\[
\mathbb{E}\bigl(g_t(s)^{\alpha L_1(t)}\bigr)= f_t
\bigl(g_t(s)^{\alpha}\bigr)= 1- \bigl(1- \bigl(1-(1-s)^{\mathit{e}^{t}}
\bigr)^\alpha \bigr)^{\mathit{e}^{-t}}.
\]
Taking the limit in $t$, and using the dominated convergence theorem
for the left-hand side, we find that the following expected value is
independent of $\alpha$:
\[
\mathbb{E}\bigl(V^\alpha\bigr)= s.
\]
This is possible only if $V$ is $\{0,1\}$-valued, equal to $1$ with
probability $s$.
Now, since $g_t(s)$ is increasing in $s$, there is a.s. a threshold
random variable
\[
U=\inf\Bigl\{s \in\mathbb Q \cap[0,1], \lim_{t \rightarrow\infty}
g_t(s)^{L_1(t)} = 1\Bigr\},
\]
which is uniformly distributed on $[0,1]$ since $\mathbb
{P}(U<s)=\mathbb{P}(V=1)=s$.
Then we form the logarithm of the expression $g_t(1-s)^{L_1(t)}$ and
use that
$\log g_t(1-s)$ is equivalent as $t \rightarrow\infty$ to
$g_t(1-s)-1$, itself
equal to $-s^{\mathit{e}^{t}}$ from the previous computation,
to deduce that
\[
\mbox{for }s <1-U, \qquad \lim_{t \rightarrow\infty} s^{\mathit{e}^{t}}
L_1(t)= 0
\]
and
\[
\mbox{for } s >1-U, \qquad  \lim_{t \rightarrow\infty}
s^{\mathit{e}^{t}} L_1(t)= \infty.
\]
Set $V=1-U$. The random variable $V$ is again uniformly distributed on $[0,1]$.
Taking again logarithm, for $\varepsilon>0$, we have
\[
-\log(V+\varepsilon) \leq\liminf_{t \rightarrow\infty} \mathrm{e}^{-t} \log
\bigl(L_1(t)\bigr) \leq \limsup_{t \rightarrow\infty}
\mathrm{e}^{-t} \log\bigl(L_1(t)\bigr) \leq-\log (V-\varepsilon),
\]
and the random variable $-\log(V)$ is exponentially distributed with
parameter $1$. This completes the proof.
\end{pf*}

The a.s. growth rate of the fixation line is the key to the following
estimate of the depth of the Bolthausen--Sznitman coalescent. There
exist other proofs in the literature, and we point the reader to the
one by Goldschmidt and Martin \cite{GM05} based on a connection with
recursive trees, and the one by
M\"ohle and Pitters \cite{MO14} based on a direct analytical approach.
Let us stress after \cite{MO14} that, once the distribution of $\tau
_1^n$ is explicitly known (which may be done using either a direct
computation or the aforementioned connection \cite{GM05}), it is a very
simple matter to derive the asymptotics of this distribution.
The interest of our approach lies in the connection with the (discrete)
Neveu branching process.


\begin{theo}
\label{depthBS}
In the Bolthausen--Sznitman case $\Lambda(dx)= \mathbf{1}_{[0,1]} (x)
  \,dx$,
we have the following convergence in distribution for the depth of the
$n$-coalescent:
%
\begin{equation}
\label{eqdepthBS}
\tau_1^n- \log\log(n) \Rightarrow- \log
\mathbf{e} \qquad \mbox{as } n \rightarrow \infty,
\end{equation}
where
$\mathbf{e}$ is an exponential random variable with parameter $1$.
\end{theo}

The random variable $-\log(\mathbf{e})$ is Gumbel distributed:
\[
\mathbb{P}\bigl(-\log(\mathbf{e}) \leq x\bigr)=\mathbb{P}\bigl(\mathbf{e}\geq
\mathrm{e}^{-x}\bigr)= \mathrm{e}^{-\mathrm{e}^{-x}},\qquad  x \in\mathbb R.
\]
The sequence $(\tau_1^n, n \geq1)$ evolves by independent exponential
jumps at the moments of records in the natural coupling; see
equation (\ref{eqrecord}). The convergence in distribution therefore
cannot be pushed to an a.s. convergence.
\begin{pf*}{Proof of Theorem~\protect\ref{depthBS}}
From the a.s. growth rate (\ref{asymptoticBS}) and the definition~(\ref{eqdeftaualpha}) of the hitting time $\alpha_1^n$,
we deduce
\[
\mathrm{e}^{-\alpha_1^n} \log n \leq \mathrm{e}^{-\alpha_1^n} \log L_1\bigl(
\alpha_1^n\bigr) \rightarrow\mathbf{e}\qquad \mbox{a.s. as } n
\rightarrow\infty,
\]
using the definition of $\alpha_1^n$ for the inequality and (\ref
{asymptoticBS}) for the almost sure convergence.
Therefore,
\[
\limsup_{n \rightarrow\infty} \log\log n -\alpha_1^n
\leq\log \mathbf{e}.
\]
Similarly, taking the left limit at $\alpha_1^n$ this time,
\[
\mathrm{e}^{-\alpha_1^n} \log n \geq \mathrm{e}^{-\alpha_1^n} \log L_1\bigl(
\alpha_1^n -\bigr) \rightarrow\mathbf{e}\qquad \mbox{a.s. as } n
\rightarrow\infty,
\]
and this implies
\[
\liminf_{n \rightarrow\infty} \log\log n-\alpha_1^n
\geq\log \mathbf{e}.
\]
This proves (\ref{eqdepthBS}) with
$\alpha_1^n$ instead of $\tau_1^n$.
We conclude using Lemma~\ref{tau-alpha}.
\end{pf*}


\begin{appendix}
\section*{Appendix}\label{app}

\begin{pf*}{Proof of Lemma~\ref{lemcombi}}
We perform the following calculations:
\begin{eqnarray*}
\tilde{\Gamma}_{i, \geq j} &=&  \sum_{k \geq j-i}
\pmatrix{k+i \cr k+1} \int_{[0,1]} \Lambda(dx) x^{-2}
x^{k+1} (1-x)^i
\\
&=& \int_{[0,1]} \Lambda(dx) x^{-2} \biggl[ \sum
_{k \geq j-i} \pmatrix{k+i \cr k+1} x^{k+1} \biggr]
(1-x)^i
\\
&=& \int_{[0,1]} \Lambda(dx) x^{-2} \biggl[
\frac{1}{(1-x)^i} - \sum_{0
\leq k \leq j-i} \pmatrix{k+i-1 \cr k}
x^{k} \biggr] (1-x)^i
\\
&=& \int_{[0,1]} \Lambda(dx) x^{-2} \biggl[ 1 - \sum
_{0 \leq k \leq
j-i} \pmatrix{k+i-1 \cr k} x^{k}
(1-x)^i \biggr]
\end{eqnarray*}
using the binomial theorem at the third equality. We also compute
\begin{eqnarray*}
\Lambda_{j, \leq i} &=& \sum_{j-i \leq k \leq j-1}
\pmatrix{j \cr k+1} \int_{[0,1]} \Lambda(dx) x^{-2} x^{k+1}
(1-x)^{j-(k+1)}
\\
&= & \int_{[0,1]} \Lambda(dx) x^{-2} \sum
_{j-i \leq k \leq j-1} \pmatrix{j \cr k+1} x^{k+1} (1-x)^{j-(k+1)}
\\
&=&  \int_{[0,1]} \Lambda(dx) x^{-2} \biggl[1- \sum
_{0 \leq k \leq j-i} \pmatrix{j \cr k} x^{k}
(1-x)^{j-k} \biggr]
\end{eqnarray*}
using the same theorem at the third equality again.
It is enough to prove that the two integrands are equal, which amounts
to verify
\[
\sum_{0 \leq k \leq j-i} \pmatrix{k+i-1 \cr k} x^{k} =
\sum_{0 \leq k \leq j-i} \pmatrix{j \cr k} x^{k}
(1-x)^{j-i-k}.
\]
But setting $\ell=j-i$ in the right-hand side, we obtain
\begin{eqnarray*}
\sum_{0 \leq k \leq j-i} \pmatrix{j \cr k} x^{k}
(1-x)^{j-i-k} &=&  \sum_{0 \leq k \leq\ell}
\pmatrix{\ell+ i \cr
k} x^{k} (1-x)^{\ell
-k}
\\
&=&  \sum_{0 \leq k + k' \leq\ell} \pmatrix{\ell+ i \cr k} \pmatrix{\ell-k \cr
k'} (-1)^{k'} x^{k+k'}.
\end{eqnarray*}
The claim therefore reduces to the following combinatorial statement:
%
\renewcommand{\theequation}{A.1}
\begin{equation}
\label{eqcombinatorial}
\pmatrix{n+i-1 \cr n} = \sum_{k + k' = n}
\pmatrix{\ell+ i \cr k} \pmatrix{\ell-k \cr k'} (-1)^{k'}\qquad \mbox{for } \ell\geq n.
\end{equation}
If $k+k'=n$, however, we have
\[
\pmatrix{\ell-k \cr k'} = (-1)^{k'} \pmatrix{n-\ell-1 \cr k'},
\]
and (\ref{eqcombinatorial}) reduces to
\[
\pmatrix{n+i-1 \cr n}= \sum_{k + k' = n} \pmatrix{\ell+ i \cr k}
\pmatrix{n-\ell-1 \cr k'}\qquad \mbox{for } \ell\geq n,
\]
a simple identity (also known as the Vandermonde identity).
\end{pf*}
\end{appendix}

\section*{Acknowledgments}

The author is grateful to Stephan Gufler, G\"otz Kersting, Iulia Stanciu, Anton Wakolbinger and
Linglong Yuan for their interest in this work. He also thanks the
referee and an Associate Editor for their careful reading.

\printaddresses
\end{document}